\documentclass[a4paper]{amsart}
\usepackage{amsthm,amsfonts,amsmath,amssymb}
\usepackage[abs]{overpic}
\usepackage{color} 
\usepackage{comment} 

\newtheorem{theorem}{Theorem}[section]
\newtheorem{lemma}[theorem]{Lemma}

\newtheorem{proposition}[theorem]{Proposition}
\newtheorem{corollary}[theorem]{Corollary}

\newtheorem{problem}[theorem]{Problem}

\theoremstyle{remark}

\theoremstyle{definition}

\newtheorem{remark}[theorem]{Remark}

\newcommand{\omu}{\overline{\mu}}
\newcommand{\e}{\varepsilon}
\newcommand{\de}{\delta}

\begin{document}
\title[The Dabkowski-Sahi invariant and $4$-moves for links]{The Dabkowski-Sahi invariant and $4$-moves for links}

\author[Haruko A. Miyazawa]{Haruko A. Miyazawa}
\address{Institute for Mathematics and Computer Science, Tsuda University,
2-1-1 Tsuda-machi, Kodaira, Tokyo 187-8577, Japan}
\curraddr{}
\email{aida@tsuda.ac.jp}
\thanks{}

\author[Kodai Wada]{Kodai Wada}
\address{Department of Mathematics, Graduate School of Science, Osaka University, Toyonaka, Osaka 560-0043, Japan}
\curraddr{}
\email{ko-wada@cr.math.sci.osaka-u.ac.jp}
\thanks{The second author was supported by JSPS KAKENHI Grant Number JP19J00006.}

\author[Akira Yasuhara]{Akira Yasuhara}
\address{Faculty of Commerce, Waseda University, 1-6-1 Nishi-Waseda, Shinjuku-ku, Tokyo 169-8050, Japan}
\curraddr{}
\email{yasuhara@waseda.jp}
\thanks{The third author was supported by JSPS KAKENHI Grant Number JP17K05264 and a Waseda University Grant for Special Research Projects (Project number: 2020C-175).}

\subjclass[2010]{Primary 57M25, 57M27; Secondary 57M05}

\keywords{$4$-move, link, welded link, link-homotopy, 
Magnus expansion}




\begin{abstract}
Dabkowski and Sahi defined an invariant of a link in the $3$-sphere, which is preserved under $4$-moves. 
This invariant is a quotient of the fundamental group of the complement of the link. 
It is generally difficult to distinguish the Dabkowski-Sahi invariants of given links. 
In this paper, we give a necessary condition for the existence of an isomorphism between the Dabkowski-Sahi invariant of a link and that of the corresponding trivial link. 
Using this condition, we provide a practical obstruction to a link to be trivial up to $4$-moves. 
\end{abstract}

\maketitle

\section{Introduction} 
A {\em $4$-move} is a local move as shown in Figure~\ref{4-move}. 
Two links are {\em $4$-equivalent} if they are related by a finite sequence of $4$-moves and isotopies. 
In the late 1970s, Nakanishi first systematically studied the $4$-moves, see~\cite{P16} for a nice survey of this topic. 
In 1979, he conjectured that every knot is $4$-equivalent to the trivial knot~\cite[Problem~1.59(3)(a)]{Kirby}. 
This conjecture remains open, although it has been verified for several classes of knots, e.g. $2$-bridge knots, closed 3-braid knots, and knots with $12$ or less crossings (cf.~\cite{DJKS,Kirby,NS}). 

\begin{figure}[htbp]
  \begin{center}
    \begin{overpic}[width=10cm]{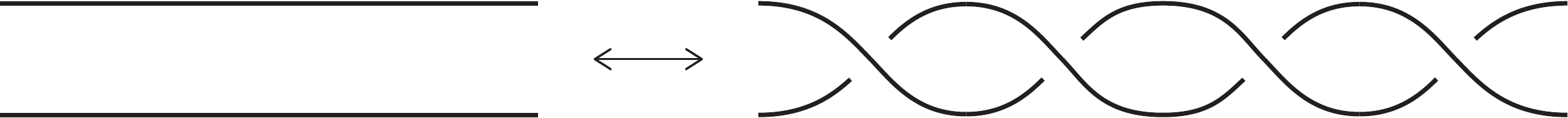}
    \end{overpic}
  \end{center}
  \caption{$4$-move}
  \label{4-move}
\end{figure}

Since the linking number modulo~$2$ is an invariant of $4$-equivalence, not every link with two or more components is $4$-equivalent to a trivial link. 
In 1985, Kawauchi proposed the following question: if two links are link-homotopic, then are they $4$-equivalent?~\cite[Problem~1.59(3)(b)]{Kirby}. 
Here, two links are {\em link-homotopic} if they are related by a finite sequence of self-crossing changes and isotopies. 
In~\cite[Theorem~3]{DP04}, Dabkowski and Przytycki gave a negative answer to the question in the case of links with three or more components by using the {\em $4$th Burnside group} of links introduced in~\cite{DP02}. 
However, the question in the case of $2$-component links still remains open. 
Since two $2$-component links are link-homotopic if and only if they have the same linking number, this open question can be rephrased as follows.

\begin{problem}[Dabkowski and Przytycki~\cite{DP04}]\label{prob}
Is it true that two $2$-component links with the same linking number are $4$-equivalent?
\end{problem}

This problem has been shown to be true for several classes of $2$-component links, e.g. $2$-algebraic links, closed $3$-braid links, links with $11$ or less crossings, and alternating links with $12$ crossings (cf.~\cite{DJKS,P16}). 
On the other hand, there is a growing belief that the problem is false. 
To give a counterexample to the problem, we need invariants of $4$-equivalence. 
While the $4$th Burnside group of links is a powerful invariant of $4$-equivalence, it cannot distinguish any $2$-component link from the trivial one~\cite[Theorem~4]{DP04} . 
This motivated Dabkowski and Sahi to define a new invariant of $4$-equivalence in~\cite{DS}. 

Given a link $L$ in the $3$-sphere $S^{3}$, Dabkowski and Sahi~\cite{DS} defined a quotient $\mathcal{R}_{4}(L)$ of the fundamental group of the complement $S^{3}\setminus L$, which is invariant under $4$-moves. 
Using this invariant $\mathcal{R}_{4}$, they reproved in~\cite[Corollary 2.19]{DS} the result of Dabkowski and Przytycki~\cite[Theorem~3]{DP04}. 
We can expect that $\mathcal{R}_{4}$ has a stronger potential to address Problem~\ref{prob} than the $4$th Burnside group of links. 
However, it is generally difficult to distinguish the invariants $\mathcal{R}_{4}$ of given links. 
Therefore it is important to find a simple method to distinguish these invariants. 
In this paper, we give a necessary condition for the existence of an isomorphism between the invariant $\mathcal{R}_{4}$ of a given link and that of the corresponding trivial link (Lemma~\ref{lem-subset}). 
This condition provides a practical obstruction to trivializing links by $4$-moves (Theorem~\ref{thm}). 

The same result of Theorem~\ref{thm} also holds for welded links. 
As its application, we show that there is a $2$-component welded link with linking number~$0$ which is not $4$-equivalent to the trivial $2$-component link (Proposition~\ref{prop-Whitehead}). 
This gives a negative answer to Problem~\ref{prob} for $2$-component welded links.

\section{The Dabkowski-Sahi invariant of links}
\label{sec-DSinv}
Throughout the paper, for a set $X$, we denote by $F(X)$ the free group on the alphabet $X^{\pm}=\{x,x^{-1}\mid x\in X\}$. 
Let $q\geq1$ be an integer. 
For a group $G$, we denote by $\Gamma_{q}G$ the $q$th term of the lower central series of $G$, and by $N_{q}G$ the $q$th nilpotent quotient $G/\Gamma_{q}G$. 
For two normal subgroups $H$ and $I$ of $G$, we denote by $H\cdot I$ the subgroup of $G$ generated by all $ab$ with $a\in H$ and $b\in I$. 

Let $L$ be an $m$-component link in $S^{3}$. 
All links in this paper are assumed to be ordered and oriented. 
We denote by $G(L)$ the fundamental group of the complement $S^{3}\setminus L$. 
Let $\langle X \mid R\rangle$ be a Wirtinger presentation of $G(L)$. 
Then we define a set of words in $F(X)$ as 
\[
S(X)=\{(awbw^{-1})^{2}(wbw^{-1}a)^{-2}~\vline~a,b\in X^{\pm},w\in F(X)\}\subset F(X). 
\]
The {\em Dabkowski-Sahi invariant} of $L$, defined in~\cite{DS}, is the group given by the presentation $\langle X\mid R,S(X)\rangle$. 
We denote it by $\mathcal{R}_{4}(L)$. 
It is shown in~\cite[Proposition~2.3]{DS} that $\mathcal{R}_{4}(L)$ is an invariant of $4$-equivalence. 

In this paper, we introduce a quotient of $\mathcal{R}_{4}(L)$. 
For an integer $n\geq1$, 
we define a set of words in $F(X)$ as 
\[
W_{n}(X)=\{w^{n}~\vline~w\in F(X)\}\subset F(X).
\] 
The {\em $n$th reduced Dabkowski-Sahi invariant} of $L$ is the group given by the presentation $\langle X\mid R,S(X),W_{n}(X)\rangle$.  
We denote it by $\mathcal{R}_{4}^{n}(L)$. 
The Dabkowski-Sahi invariant $\mathcal{R}_{4}(L)$ and the reduced one $\mathcal{R}_{4}^{n}(L)$ are not always finite, but the nilpotent quotient $N_{q}\mathcal{R}_{4}^{n}(L)$ is always a finite group for any $q\geq1$ (cf.~\cite[Chapter~2]{V}). 
Since $\mathcal{R}_{4}(L)$ is invariant under $4$-moves, 
the following result is obtained immediately. 

\begin{proposition}
The group $N_{q}\mathcal{R}_{4}^{n}(L)$ and its cardinality $|N_{q}\mathcal{R}_{4}^{n}(L)|$ are invariant under $4$-moves.
\end{proposition}

Now consider a diagram of $L$. 
For each $i\in\{1,\ldots, m\}$, we choose one arc of the $i$th component of the diagram and denote it by $x_{i1}$. 
As shown in Figure~\ref{schematic}, let $x_{i2},x_{i3},\ldots,x_{ir_{i}}$ be the other arcs of the $i$th component in turn with respect to the orientation, 
where $r_{i}$ denotes the number of arcs of the $i$th component. 
In the figure, $u_{ij}\in\{x_{kl}\}$ denotes the arc which separates $x_{ij}$ and $x_{ij+1}$. 
Let $\e_{ij}$ be the sign of the crossing among $x_{ij},u_{ij}$ and $x_{ij+1}$. 
%
For $1\leq i\leq m$ and $1\leq j\leq r_{i}$, we put 
\[
v_{ij}=u_{i1}^{\e_{i1}}u_{i2}^{\e_{i2}}\cdots u_{ij}^{\e_{ij}}\in F(X), 
\] 
where $X=\{x_{ij}\mid 1\leq i\leq m,\,1\leq j\leq r_{i}\}$. 
In particular, we denote by $\lambda_{i}$ the word $v_{ir_{i}}$. 
By the geometric construction of the Wirtinger presentation, 
the word $\lambda_{i}$ corresponds to an $i$th longitude of $L$. 
In this sense, we call $\lambda_{i}$ an {\em $i$th longitude word} of $L$. 
Let $A=\{\alpha_{1},\alpha_{2},\ldots,\alpha_{m}\}$. 
Following~\cite{M57}, we define a sequence of homomorphisms $\eta_{q}:F(X)\rightarrow F(A)$ inductively by 
\begin{eqnarray*}
\eta_{1}(x_{ij})=\alpha_{i},& \\ 
\eta_{q+1}(x_{i1})=\alpha_{i}, & 
\mbox{and} \hspace{1em}
\eta_{q+1}(x_{ij+1})=\eta_{q}(v_{ij}^{-1})\,\alpha_{i}\,\eta_{q}(v_{ij}). 
\end{eqnarray*}

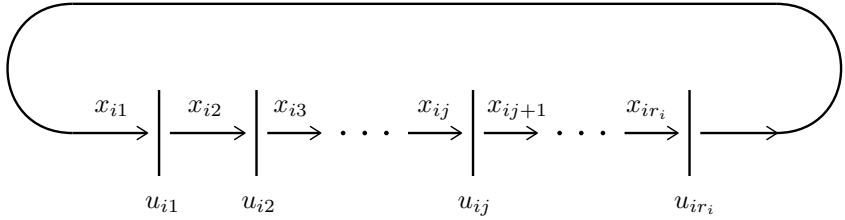
\begin{figure}[htb]
  \begin{center}
    \begin{overpic}[width=11cm]{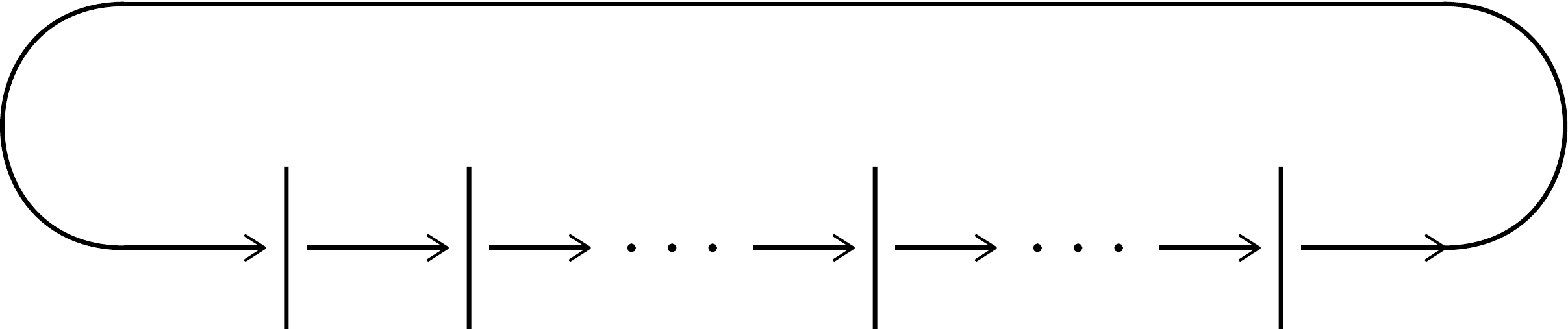}
      \put(33,24){$x_{i1}$}
      \put(68,24){$x_{i2}$}
      \put(100,24){$x_{i3}$}
      \put(154,24){$x_{ij}$}
      \put(180,24){$x_{ij+1}$}
      \put(232,24){$x_{ir_{i}}$}
      \put(52,-12){$u_{i1}$}
      \put(88,-12){$u_{i2}$}
      \put(169,-12){$u_{ij}$}
      \put(250,-12){$u_{ir_{i}}$}
    \end{overpic}
  \end{center}
  \vspace{1em}
  \caption{A schematic illustration of the $i$th component}
  \label{schematic}
\end{figure}

In~\cite[Theorem~4]{M57}, Milnor used the homomorphism $\eta_{q}$ to give an explicit presentation of the nilpotent quotient $N_{q}G(L)$ as follows: 
\[
N_{q}G(L)
\cong \langle \alpha_{1},\ldots,\alpha_{m}\mid [\alpha_{1},\eta_{q}(\lambda_{1})], \ldots, [\alpha_{m},\eta_{q}(\lambda_{m})], \Gamma_{q}F(A)\rangle. 
\] 
For the groups $N_{q}\mathcal{R}_{4}(L)$ and $N_{q}\mathcal{R}_{4}^{n}(L)$, we obtain the following presentations. 

\begin{lemma}\label{lem-pres}
The group $N_{q}\mathcal{R}_{4}(L)$ has the presentation 
\[
\langle \alpha_{1},\ldots,\alpha_{m}\mid [\alpha_{1},\eta_{q}(\lambda_{1})], \ldots, [\alpha_{m},\eta_{q}(\lambda_{m})], S(A), \Gamma_{q}F(A)\rangle, 
\]
and the group $N_{q}\mathcal{R}_{4}^{n}(L)$ has the presentation
\[
\langle \alpha_{1},\ldots,\alpha_{m}\mid [\alpha_{1},\eta_{q}(\lambda_{1})], \ldots, [\alpha_{m},\eta_{q}(\lambda_{m})], S(A), W_{n}(A), \Gamma_{q}F(A)\rangle. 
\] 
\end{lemma}

The proof of this lemma is a minor modification of that of~\cite[Theorem~4 ]{M57}, and hence we omit it.

\begin{lemma}\label{lem-subset}
Let $L$ be an $m$-component link and $\lambda_{i}$ its $i$th longitude word $(1\leq i\leq m)$. 
Let $O$ be the trivial $m$-component link. 
If $\mathcal{R}_{4}(L)$ and $\mathcal{R}_{4}(O)$ are isomorphic, 
then the set $\{\left[\alpha_{i},\eta_{q}(\lambda_{i})\right]~\vline~1\leq i\leq m\}$ is a subset of the normal subgroup $\langle\langle S(A)\rangle\rangle\cdot \langle\langle W_{n}(A)\rangle\rangle\cdot\Gamma_{q}F(A)$ of $F(A)$, 
where $\langle\langle\ \cdot\ \rangle\rangle$ denotes the normal closure. 
\end{lemma}

\begin{proof}
The proof is parallel to that of \cite[Lemma~2.1]{MWY}. 
By Lemma~\ref{lem-pres}, we have 
\[
N_{q}\mathcal{R}_{4}^{n}(L)\cong
\langle \alpha_{1},\ldots,\alpha_{m}\mid [\alpha_{1},\eta_{q}(\lambda_{1})], \ldots, [\alpha_{m},\eta_{q}(\lambda_{m})], S(A), W_{n}(A), \Gamma_{q}F(A)\rangle 
\]
and
\[
N_{q}\mathcal{R}_{4}^{n}(O)\cong
\langle \alpha_{1},\ldots,\alpha_{m}\mid S(A), W_{n}(A), \Gamma_{q}F(A)\rangle. 
\]
Consider the following sequence of two natural projections $\psi$ and $\phi$: 
\begin{eqnarray*}
F(A)&\overset{\psi}{\longrightarrow}&
\langle \alpha_{1},\ldots,\alpha_{m}\mid S(A), W_{n}(A), \Gamma_{q}F(A)\rangle \\
&\overset{\phi}{\longrightarrow}&
\langle \alpha_{1},\ldots,\alpha_{m}\mid [\alpha_{1},\eta_{q}(\lambda_{1})], \ldots, [\alpha_{m},\eta_{q}(\lambda_{m})], S(A), W_{n}(A), \Gamma_{q}F(A)\rangle. 
\end{eqnarray*}
Then it follows that $\psi([\alpha_{i},\eta_{q}(\lambda_{i})])\in\ker{\phi}$ for any $i\in\{1,\ldots, m\}$. 
On the other hand, we have 
\[
|N_{q}\mathcal{R}_{4}^{n}(O)|
=|\langle \alpha_{1},\ldots,\alpha_{m}\mid S(A), W_{n}(A), \Gamma_{q}F(A)\rangle|
=|N_{q}\mathcal{R}_{4}^{n}(L)|
\times\left|\ker{\phi}\right|. 
\]
Since $\mathcal{R}_{4}(L)$ and $\mathcal{R}_{4}(O)$ are isomorphic, $N_{q}\mathcal{R}_{4}^{n}(L)$ and $N_{q}\mathcal{R}_{4}^{n}(O)$ are also isomorphic, and furthermore, they are finite groups. 
This implies that $|\ker{\phi}|$ must be equal to~$1$. 
Hence $[\alpha_{i},\eta_{q}(\lambda_{i})]\in\langle\langle S(A)\rangle\rangle\cdot \langle\langle W_{n}(A)\rangle\rangle\cdot\Gamma_{q}F(A)$. 
\end{proof}

\section{Obstruction to trivializing links by $4$-moves} 
Denote by $\mathbb{Z}_{2}$ the cyclic group $\mathbb{Z}/2\mathbb{Z}$ of order $2$. 
The {\em Magnus $\mathbb{Z}_{2}$-expansion $E^{(2)}$} is a homomorphism from $F(A)$ to the ring of the formal power series in non-commutative variables $X_{1},\ldots,X_{m}$ with coefficients in $\mathbb{Z}_{2}$ 
defined by
\[
E^{(2)}(\alpha_{i})=1+X_{i}
\hspace{1em} \mbox{and} \hspace{1em}
E^{(2)}(\alpha_{i}^{-1})=1+X_{i}+X_{i}^{2}+X_{i}^{3}+\cdots 
\]
for $1\leq i\leq m$.

\begin{theorem}\label{thm}
Let $n\geq8$ be a power of $2$ and $q\geq5$ an integer. 
Let $L$ be an $m$-component link and $\lambda_{i}$ its $i$th longitude word $(1\leq i\leq m)$. 
Assume that 
\[
E^{(2)}\left(\left[\alpha_{i},\eta_{q}(\lambda_{i})\right]\right)=1+
\sum_{j\geq 1}\sum_{\substack{k_{1},\ldots,k_{j}\\ \hspace{1em}\in\{1,\ldots, m\}}} 
c(k_{1},\ldots,k_{j})X_{k_{1}}\cdots X_{k_{j}}
\] 
for certain coefficients $c(k_{1},\ldots,k_{j})\in\mathbb{Z}_{2}$.
If $L$ is $4$-equivalent to the trivial $m$-component link, 
then the following hold.

\begin{enumerate}
\item $c(k)=c(k,k)=c(k,k,k)=c(k,k,k,k)=0$ for any $k$. 
\item $c(k,l,l)=c(l,k,k)=c(k,k,l)=c(l,l,k)$ for any $k\neq l$. 
\item $c(k,l,k,l)=c(l,k,l,k)$, $c(k,k,l,l)=c(l,l,k,k)$, and $c(k,k,k,l)=c(k,k,l,k)$ \\
$=c(k,l,k,k)=c(l,k,k,k)$ for any $k\neq l$. 
\item If $c(k_{1},k_{2},k_{3},k_{4})\neq 0$ and the cardinality $|\{k_{1},k_{2},k_{3},k_{4}\}|=3$, then $k_{1}=k_{2}$, $k_{2}=k_{3}$, or $k_{3}=k_{4}$ for any $k_{1},k_{2},k_{3},k_{4}$. 
\item $c(k_{1},\ldots,k_{j})=0$ for any pairwise distinct indices $k_{1},\ldots,k_{j}$ with $j<\min\{n,q\}$.\footnote{
Since $n$ and $q$ can be chosen arbitrarily large, the condition $j<\min\{n,q\}$ is no real restriction.} 
\end{enumerate}
\end{theorem}

To show this theorem, we prepare the following two lemmas. 

\begin{lemma}\label{lem-power}
Let $n$ be a positive power of $2$. 
If $t\in \langle\langle W_{n}(A)\rangle\rangle$, then $E^{(2)}(t)=1+\mathcal{D}(n)$, 
where $\mathcal{D}(n)$ denotes the terms of degree $\geq n$. 
\end{lemma}

\begin{proof}
Since $\langle\langle W_{n}(A)\rangle\rangle$ is a verbal subgroup, 
the element $t$ can be written in the form $\prod_{j}w_{j}^{n}$ for some $w_{j}\in F(A)$. 
For each $w_{j}$, the Magnus $\mathbb{Z}_{2}$-expansion $E^{(2)}\left(w_{j}\right)$ has the form $1+\mathcal{D}(1)$. 
Since $n$ is a power of $2$, all the binomial coefficients $\binom{n}{r}$ are even for $0<r<n$. 
Therefore 
\[E^{(2)}\left(w_{j}^{n}\right)
=\left(1+\mathcal{D}(1)\right)^{n}
=1+\left(\mathcal{D}(1)\right)^{n}
=1+\mathcal{D}(n)
\] 
for each $j$. 
This implies that $E^{(2)}(t)=1+\mathcal{D}(n)$. 
\end{proof}

\begin{lemma}
\label{lem-lcs}
If $u\in \Gamma_{q}F(A)$, then $E^{(2)}(u)=1+\mathcal{D}(q)$.
\end{lemma}

\begin{proof}
This follows from~\cite[Corollary~5.7]{MKS} directly.  
\end{proof}

\begin{proof}[Proof of Theorem~\ref{thm}]
Since $L$ is $4$-equivalent to the trivial $m$-component link $O$, $\mathcal{R}_{4}(L)$ and $\mathcal{R}_{4}(O)$ are isomorphic. 
By Lemma~\ref{lem-subset}, we have 
\[
\{\left[\alpha_{i},\eta_{q}(\lambda_{i})\right]\mid 1\leq i\leq m\}
\subset\langle\langle S(A)\rangle\rangle\cdot \langle\langle W_{n}(A)\rangle\rangle \cdot\Gamma_{q}F(A).
\] 
This implies that each $\left[\alpha_{i},\eta_{q}(\lambda_{i})\right]$ can be written in the form $stu$ for some $s\in\langle\langle S(A)\rangle\rangle$, $t\in \langle\langle W_{n}(A)\rangle\rangle$, and $u\in\Gamma_{q}F(A)$. 
By Lemmas~\ref{lem-power} and~\ref{lem-lcs}, we have
\[
E^{(2)}(t)=1+\mathcal{D}(n) 
\hspace{1em} \mbox{and} \hspace{1em}
E^{(2)}(u)=1+\mathcal{D}(q).
\]
Since $\min\{n,q\}\geq5$, it is enough to show that $E^{(2)}(s)$ satisfies (i)--(v). 

The element $s\in \langle\langle S(A)\rangle\rangle$ can be written in the form 
$\prod_{j}g_{j}^{-1}s_{j}^{\e_{j}}g_{j}$ for some $g_{j}\in F(A)$, $s_{j}\in S(A)$, and $\e_{j}\in\{1,-1\}$.
Put $E^{(2)}\left(s_{j}^{\e_{j}}\right)=1+S$. 
The Magnus $\mathbb{Z}_{2}$-expansions $E^{(2)}(g_{j})$ and $E^{(2)}(g_{j}^{-1})$ have the same terms of degree~$1$. 
We denote the terms by $G_{1}$. 
By a straightforward computation, it follows that 
\begin{equation}\label{eq-1}
E^{(2)}(g_{j}^{-1}s_{j}^{\e_{j}}g_{j})
=1+S+SG_{1}+G_{1}S+\mathcal{D}(\mbox{(minimal degree of $S$)}+2). 
\end{equation}

Now we consider $E^{(2)}(s_{j}^{\e_{j}})$. 
The element $s_{j}^{\e_{j}}$ can be written in the form 
\[
\left(\alpha_{k}^{\e}w\alpha_{l}^{\de}w^{-1}\right)^{2}\left(w\alpha_{l}^{\de}w^{-1}\alpha_{k}^{\e}\right)^{-2}
\] 
for some $w\in F(A)$, $k,l\in\{1,\ldots,m\}$, and $\e,\de\in\{1,-1\}$.
We define a formal power series $f_{r}^{\e}$ as 
\[
f_{r}^{\e}=
\begin{cases}
0 & (\e=1), \\
X_{r}^{2}+X_{r}^{3}+X_{r}^{4}+\cdots &(\e=-1).
\end{cases}
\] 
Putting $E^{(2)}(w)=1+F$ and $E^{(2)}(w^{-1})=1+\overline{F}$, 
we have
\begin{eqnarray*}
&&E^{(2)}\left(\alpha_{k}^{\e}w\alpha_{l}^{\de}w^{-1}\right) \\
&&=\left(1+X_{k}+f_{k}^{\e}\right)(1+F)\left(1+X_{l}+f_{l}^{\de}\right)(1+\overline{F}) \\
&&=\left(1+X_{k}+f_{k}^{\e}\right)
\left(1+X_{l}+X_{l}\overline{F}+f_{l}^{\de}+f_{l}^{\de}\overline{F}+FX_{l}+FX_{l}\overline{F}+Ff_{l}^{\de}+Ff_{l}^{\de}\overline{F}\right) \\
&&=1+X_{k}+X_{l}+P_{1}, 
\end{eqnarray*}
where $P_{1}=X_{k}X_{l}+f_{k}^{\e}+f_{l}^{\de}+FX_{l}+X_{l}\overline{F}+X_{k}FX_{l}+X_{k}X_{l}\overline{F}+FX_{l}\overline{F}+X_{k}f_{l}^{\de}+f_{k}^{\e}X_{l}+Ff_{l}^{\de}+f_{l}^{\de}\overline{F}+\mathcal{D}(4)$. 
Similarly we have 
\[
E^{(2)}\left(\left(w\alpha_{l}^{\de}w^{-1}\alpha_{k}^{\e}\right)^{-1}\right) 
=E^{(2)}\left(\alpha_{k}^{-\e}w\alpha_{l}^{-\de}w^{-1}\right) 
=1+X_{k}+X_{l}+P_{2}, 
\]
where $P_{2}$ is obtained from $P_{1}$ by replacing $\e$ and $\de$ with $-\e$ and $-\de$, respectively. 
Then it follows that 
\begin{eqnarray*}
E^{(2)}(s_{j}^{\e_{j}})
&=&\left(1+X_{k}+X_{l}+P_{1}\right)^{2}\left(1+X_{k}+X_{l}+P_{2}\right)^{2} \\
&=&\left(1+(X_{k}+X_{l})^{2}+\left(X_{k}+X_{l}\right)P_{1}+P_{1}\left(X_{k}+X_{l}\right)+P_{1}^{2}\right) \\
&& 
\times\left(1+(X_{k}+X_{l})^{2}+\left(X_{k}+X_{l}\right)P_{2}+P_{2}\left(X_{k}+X_{l}\right)+P_{2}^{2}\right). 
\end{eqnarray*}
For $r=1,2$ we put $Q_{r}=(X_{k}+X_{l})P_{r}+P_{r}(X_{k}+X_{l})+P_{r}^{2}$~$(\in\mathcal{D}(3))$ and therefore have 
\begin{equation}\label{eq-2}
E^{(2)}(s_{j}^{\e_{j}})
=1+(X_{k}+X_{l})^{4}+Q_{1}+Q_{2}+\mathcal{D}(5). 
\end{equation}
Here we observe $Q_{1}+Q_{2}=(X_{k}+X_{l})(P_{1}+P_{2})+(P_{1}+P_{2})(X_{k}+X_{l})+P_{1}^{2}+P_{2}^{2}$. 
First we focus on $(X_{k}+X_{l})(P_{1}+P_{2})+(P_{1}+P_{2})(X_{k}+X_{l})$. 
Since $f_{k}^{\e}+f_{k}^{-\e}=X_{k}^{2}+X_{k}^{3}+\cdots$ and $f_{l}^{\de}+f_{l}^{-\de}=X_{l}^{2}+X_{l}^{3}+\cdots$, we have 
\[
P_{1}+P_{2}=X_{k}^{2}+X_{l}^{2}+X_{k}^{3}+X_{l}^{3}+X_{k}X_{l}^{2}+X_{k}^{2}X_{l}+F_{1}X_{l}^{2}+X_{l}^{2}F_{1}+\mathcal{D}(4), 
\] 
where $F_{1}$ denotes the terms of degree $1$ in $F$. 
Note that $F_{1}$ also denotes the terms of degree $1$ in $\overline{F}$. 
This implies that 
\begin{eqnarray*}
&&(X_{k}+X_{l})(P_{1}+P_{2})+(P_{1}+P_{2})(X_{k}+X_{l}) \\
&&=X_{k}X_{l}^{2}+X_{l}X_{k}^{2}+X_{k}^{2}X_{l}+X_{l}^{2}X_{k} \\
&&\hspace{1em}
+X_{l}X_{k}^{3}+X_{l}X_{k}^{2}X_{l}+X_{l}X_{k}X_{l}^{2}+X_{k}^{2}X_{l}X_{k}+X_{k}X_{l}^{2}X_{k}+X_{l}^{3}X_{k} 
\\
&&\hspace{1em}
+X_{k}F_{1}X_{l}^{2}+X_{k}X_{l}^{2}F_{1}+X_{l}F_{1}X_{l}^{2}+X_{l}^{3}F_{1} \\
&&\hspace{1em}
+F_{1}X_{l}^{2}X_{k}+F_{1}X_{l}^{3}+X_{l}^{2}F_{1}X_{k}+X_{l}^{2}F_{1}X_{l}
+\mathcal{D}(5).  
\end{eqnarray*}
Next we focus on $P_{1}^{2}+P_{2}^{2}$. 
It follows that 
\begin{eqnarray*}
P_{1}^{2}&=&(
X_{k}X_{l}+F_{1}X_{l}+X_{l}F_{1}+f_{k}^{\e}+f_{l}^{\de}+\mathcal{D}(3))^{2} \\
&=&(X_{k}X_{l}+F_{1}X_{l}+X_{l}F_{1})^{2}
+(X_{k}X_{l}+F_{1}X_{l}+X_{l}F_{1})(f_{k}^{\e}+f_{l}^{\de}) \\
&&+(f_{k}^{\e}+f_{l}^{\de})(X_{k}X_{l}+F_{1}X_{l}+X_{l}F_{1})+(f_{k}^{\e}+f_{l}^{\de})^{2}+\mathcal{D}(5) 
\end{eqnarray*}
and 
\begin{eqnarray*}
P_{2}^{2}&=&(X_{k}X_{l}+F_{1}X_{l}+X_{l}F_{1})^{2}
+(X_{k}X_{l}+F_{1}X_{l}+X_{l}F_{1})(f_{k}^{-\e}+f_{l}^{-\de}) \\
&&+(f_{k}^{-\e}+f_{l}^{-\de})(X_{k}X_{l}+F_{1}X_{l}+X_{l}F_{1})+(f_{k}^{-\e}+f_{l}^{-\de})^{2}+\mathcal{D}(5). 
\end{eqnarray*}
This implies that 
\begin{eqnarray*}
&&P_{1}^{2}+P_{2}^{2} \\
&&=(X_{k}X_{l}+F_{1}X_{l}+X_{l}F_{1})(X_{k}^{2}+X_{l}^{2}) +(X_{k}^{2}+X_{l}^{2})(X_{k}X_{l}+F_{1}X_{l}+X_{l}F_{1}) \\
&&\hspace{1em}
+\mathcal{D}(5)
+X_{k}^{4}+X_{l}^{4}
+\begin{cases}
X_{k}^{2}X_{l}^{2}+X_{l}^{2}X_{k}^{2} & (\e\de=1), \\
0 & (\e\de=-1). 
\end{cases}
\end{eqnarray*}
Therefore we have 
\begin{eqnarray}\label{eq-3} 
Q_{1}+Q_{2}
&=&X_{k}X_{l}^{2}+X_{l}X_{k}^{2}+X_{k}^{2}X_{l}+X_{l}^{2}X_{k} \\ \nonumber
&&
+X_{k}X_{l}X_{k}X_{l}+X_{l}X_{k}X_{l}X_{k}+I_{kl}(\e,\de)+(X_{k}+X_{l})^{4} \\ \nonumber
&&
+X_{k}F_{1}X_{l}^{2}+X_{k}X_{l}^{2}F_{1}+F_{1}X_{l}^{2}X_{k}+X_{l}^{2}F_{1}X_{k} \\ \nonumber
&&
+F_{1}X_{l}X_{k}^{2}+X_{l}F_{1}X_{k}^{2}+X_{k}^{2}F_{1}X_{l}+X_{k}^{2}X_{l}F_{1}
+\mathcal{D}(5),
\end{eqnarray}
where 
\[
I_{kl}(\e,\de)=
\begin{cases}
0 & (\e\de=1), \\
X_{k}^{2}X_{l}^{2}+X_{l}^{2}X_{k}^{2} & (\e\de=-1). 
\end{cases}
\]
By Equations~(\ref{eq-2}) and~(\ref{eq-3}), we have 
\begin{eqnarray*}
E^{(2)}\left(s_{j}^{\e_{j}}\right)&=&1+X_{k}X_{l}^{2}+X_{l}X_{k}^{2}+X_{k}^{2}X_{l}+X_{l}^{2}X_{k} \\
&&
+X_{k}X_{l}X_{k}X_{l}+X_{l}X_{k}X_{l}X_{k}+I_{kl}(\e,\de) \\
&&
+X_{k}F_{1}X_{l}^{2}+X_{k}X_{l}^{2}F_{1}+F_{1}X_{l}^{2}X_{k}+X_{l}^{2}F_{1}X_{k} \\
&&
+F_{1}X_{l}X_{k}^{2}+X_{l}F_{1}X_{k}^{2}+X_{k}^{2}F_{1}X_{l}+X_{k}^{2}X_{l}F_{1}
+\mathcal{D}(5). 
\end{eqnarray*}

Applying this to Equation~(\ref{eq-1}), it follows that 
\begin{eqnarray*}
E^{(2)}\left(g_{j}^{-1}s_{j}^{\e_{j}}g_{j}\right)&=&1+X_{k}X_{l}^{2}+X_{l}X_{k}^{2}+X_{k}^{2}X_{l}+X_{l}^{2}X_{k} \\
&&
+X_{k}X_{l}X_{k}X_{l}+X_{l}X_{k}X_{l}X_{k}+I_{kl}(\e,\de) \\
&&
+X_{k}F_{1}X_{l}^{2}+X_{k}X_{l}^{2}F_{1}+F_{1}X_{l}^{2}X_{k}+X_{l}^{2}F_{1}X_{k} \\
&&
+F_{1}X_{l}X_{k}^{2}+X_{l}F_{1}X_{k}^{2}+X_{k}^{2}F_{1}X_{l}+X_{k}^{2}X_{l}F_{1} \\
&&
+X_{k}X_{l}^{2}G_{1}+X_{l}X_{k}^{2}G_{1}+X_{k}^{2}X_{l}G_{1}+X_{l}^{2}X_{k}G_{1} \\
&&
+G_{1}X_{k}X_{l}^{2}+G_{1}X_{l}X_{k}^{2}+G_{1}X_{k}^{2}X_{l}+G_{1}X_{l}^{2}X_{k}
+\mathcal{D}(5).
\end{eqnarray*}
This shows that $E^{(2)}\left(g_{j}^{-1}s_{j}^{\e_{j}}g_{j}\right)$ satisfies (i)--(iv) for each $j$, and hence $E^{(2)}(s)=E^{(2)}\left(\prod_{j}g_{j}^{-1}s_{j}^{\e_{j}}g_{j}\right)$ also satisfies them. 

Let $\mathcal{O}(2)$ be the ideal generated by monomials containing $X_{r}$ at least twice for some $r~(=1,\ldots, m)$. 
Then we have 
\begin{eqnarray*}
&&E^{(2)}\left(\alpha_{k}^{\e}w\alpha_{l}^{\de}w^{-1}\right) \\
&&\equiv (1+X_{k})(1+F)(1+X_{l})\left(1+\overline{F}\right)
\equiv E^{(2)}\left((w\alpha_{l}^{\de}w^{-1}\alpha_{k}^{\e})^{-1}\right)
\pmod{\mathcal{O}(2)}. 
\end{eqnarray*}
This implies that 
\[
E^{(2)}(s_{j}^{\e_{j}})
\equiv\left((1+X_{k})(1+F)(1+X_{l})\left(1+\overline{F}\right)\right)^{4}
\equiv1 \pmod{\mathcal{O}(2)}.
\]
Consequently, $E^{(2)}(s)-1\in\mathcal{O}(2)$, and hence satisfies (v). 
\end{proof}

\begin{remark}\label{rem-Milnor}
Given a diagram $D$ of an $m$-component link $L$ and a sequence $I$ of elements in $\{1,\ldots,m\}$, Milnor~\cite{M54,M57} defined an integer $\mu_{D}(I)$ 
and proved that the residue class $\omu_{L}(I)$ of $\mu_{D}(I)$, modulo an integer $\Delta_{D}(I)$ determined by the $\mu$, is an invariant of the link $L$. 
Let $\lambda_{i}$ be an $i$th longitude word obtained from~$D$. 
If $\Delta_{D}(k_{1}\ldots k_{j}\,i)=0$, then by definition $\mu_{D}(k_{1}\ldots k_{j}\,i)$ modulo~$2$ coincides with the coefficient of $X_{k_{1}}\cdots X_{k_{j}}$ in $E^{(2)}(\alpha_{i}^{-w_{i}}\eta_{q}(\lambda_{i}))$, 
where $w_{i}$ denotes the sum of signs of all self-crossings of the $i$th component of $D$. 
Hence, a straightforward computation gives that the coefficient of $X_{i}X_{k_{1}}\cdots X_{k_{j}}$ in $E^{(2)}([\alpha_{i},\eta_{q}(\lambda_{i})])=E^{(2)}([\alpha_{i},\alpha_{i}^{-w_{i}}\eta_{q}(\lambda_{i})])$ coincides with $\mu_{D}(k_{1}\ldots k_{j}\,i)$ modulo~$2$. 
\end{remark}

\begin{proposition}\label{prop-ML}
Let $L_{m}$ be the $m$-component link represented by a diagram as shown in Figure~\ref{ML}. 
Then $L_{m}$ and the trivial $m$-component link are not $4$-equivalent. 
\end{proposition}

\begin{proof}
Let $D_{m}$ be the diagram in Figure~\ref{ML} and 
$\sigma$ any permutation of the set $\{1,2,\ldots,m-2\}$. 
It follows from~\cite[Section~5]{M54} that 
$\Delta_{D_{m}}(I)=0$ for any sequence~$I$ with length~$\leq m$ and 
\[
\omu_{L_{m}}(\sigma(1)\ldots\sigma(m-2)\,m-1\,m)=
\mu_{D_{m}}(\sigma(1)\ldots\sigma(m-2)\,m-1\,m)=
\begin{cases}
1 & (\sigma=\mbox{id}), \\
0 & (\mbox{otherwise}).
\end{cases}
\]
Let $\lambda_{m}$ be an $m$th longitude word obtained from $D_{m}$. 
By Remark~\ref{rem-Milnor}, 
the coefficient of $X_{m}X_{1}X_{2}\cdots X_{m-1}$ in $E^{(2)}([\alpha_{m},\eta_{q}(\lambda_{m})])$ is equal to $1$ for $\min{\{n,q\}}\geq~m$. 
Therefore Theorem~\ref{thm}(v) completes the proof. 
\end{proof}

\begin{figure}[htbp]
  \begin{center}
    \begin{overpic}[width=9cm]{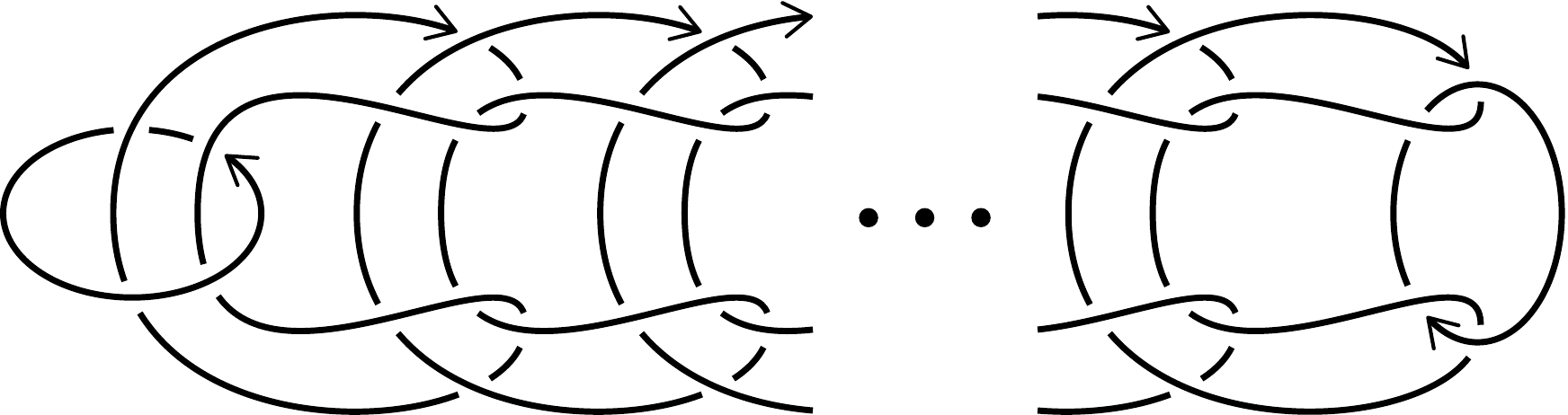}
      \put(2,14){$m$}
      \put(52,-12){$1$}
      \put(94,-12){$2$}
      \put(127,-12){$3$}
      \put(205,-12){$m-2$}
      \put(257,14){$m-1$}
    \end{overpic}
  \end{center}
  \vspace{1em}
  \caption{}
  \label{ML}
\end{figure}

\begin{remark}
(i)~
In~\cite{N}, Nakanishi showed Proposition~\ref{prop-ML} only for the case~$m=3$. 
(Note that $L_{3}$ is the Borromean rings.)
In~\cite[Corollary~2.3]{DS}, Dabkowski and Sahi obtained the same result by using the invariant $\mathcal{R}_{4}$. 

\noindent
(ii)~
Proposition~\ref{prop-ML} also follows from~\cite[Proposition~6.2]{MWY-2nlh} directly. 
Moreover, it follows that $L_{m}$ and the trivial $m$-component link cannot be deformed into each other by a finite sequence of $4$-moves and link-homotopies. 
\end{remark}

\section{Welded links}
An {\em $m$-component virtual link diagram} is the image of an immersion of $m$ circles into the plane, whose singularities are only transverse double points. 
Such double points are divided into {\em classical crossings} and {\em virtual crossings}. 
{\em Welded Reidemeister moves} consist of eight types of local moves. 
An {\em $m$-component welded link} is an equivalence class of $m$-component virtual link diagrams under welded Reidemeister moves (cf.~\cite{FRR}). 

An {\em arc} of a virtual link diagram $D$ is a segment along $D$ which goes from a classical under-crossing to the next one, 
where virtual crossings are ignored. 
The {\em group} of $D$ is defined by the Wirtinger presentation, i.e. 
each arc of $D$ yields a generator and each classical crossing gives a relation. 
Note that no new generators or relation are added at a virtual crossing.
It is easily shown that the group of $D$ is an invariant of welded links (cf.~\cite[Section~4]{Kauffman}). 
The {\em group} of a welded link is given by the group for any virtual link diagram of the welded link. 
Given a welded link $L$, we similarly define the Dabkowski-Sahi invariant $\mathcal{R}_{4}(L)$ and the reduced one $\mathcal{R}_{4}^{n}(L)$, via the group of $L$. 
Clearly, the same result of Theorem~\ref{thm} also holds for welded links. 
Even when applying it, we will refer Theorem~\ref{thm}. 

The {\em $(i,j)$-linking number $\mbox{{\rm lk}}_{i/j}$} of a welded link is the sum of signs, in a diagram for the welded link, of all classical crossings where the $i$th component passes over the $j$th one (cf.~\cite[Section~1.7]{GPV}). 
For classical links in $S^{3}$, the half of $\mbox{lk}_{i/j}+\mbox{lk}_{j/i}$ coincides with the usual linking number between the $i$th and $j$th components.

\begin{proposition}\label{prop-Whitehead}
There is a $2$-component welded link with $\mbox{{\rm lk}}_{1/2}=\mbox{{\rm lk}}_{2/1}=0$ such that the welded link and the trivial $2$-component link are not $4$-equivalent. 
\end{proposition}

\begin{proof}
Let $L$ be the $2$-component welded link represented by a virtual link diagram as shown in Figure~\ref{Whitehead}. 
By definition, $\mbox{{\rm lk}}_{1/2}=\mbox{{\rm lk}}_{2/1}=0$. 
Now consider the arcs $x_{ij}$ of the diagram in Figure~\ref{Whitehead}. 
Since $\lambda_{1}=x_{21}^{-1}x_{21}$ and $\lambda_{2}=x_{11}x_{12}^{-1}$, 
we have 
\[
[\alpha_{2},\eta_{5}(\lambda_{2})]=
\alpha_{2}^{-1}\alpha_{1}\alpha_{2}^{-1}\alpha_{1}^{-1}
\alpha_{2}\alpha_{1}\alpha_{2}\alpha_{1}^{-1}. 
\]
After computing $E^{(2)}(\left[\alpha_{2},\eta_{5}(\lambda_{2})\right])$, we see that the coefficient of $X_{1}X_{2}X_{2}$ is $1$ and that of $X_{2}X_{1}X_{1}$ is $0$. 
By Theorem~\ref{thm}(ii), $L$ is not $4$-equivalent to the trivial $2$-component link. 
\end{proof}

\begin{figure}[htbp]
  \begin{center}
    \begin{overpic}[width=3.5cm]{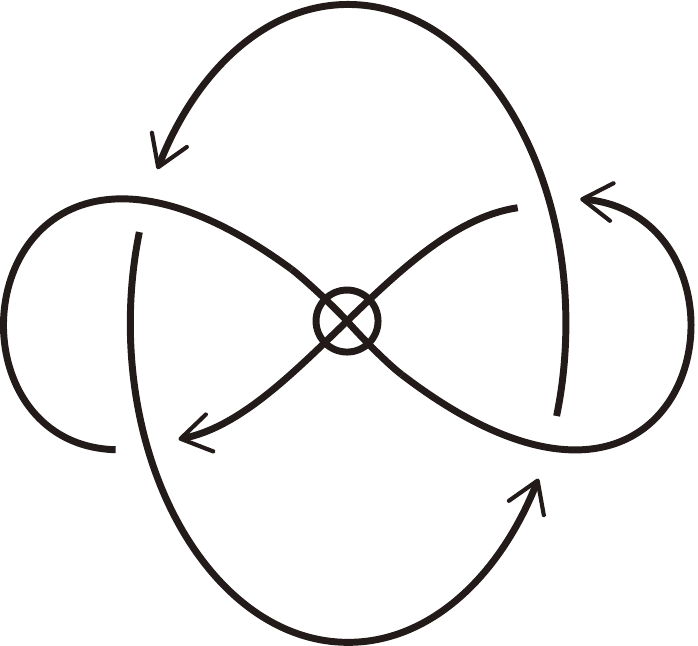}
      \put(-16,45){$x_{21}$}
      \put(51,66){$x_{22}$}
      \put(72,86){$x_{11}$}
      \put(18,2){$x_{12}$}
    \end{overpic}
  \end{center}
  \caption{}
  \label{Whitehead}
\end{figure}

\begin{remark}
The $4$th Burnside group of classical links, defined by Dabkowski and Przytycki in~\cite{DP02}, naturally extends to welded links. 
Although we do not give here the precise definition,
one can easily verify that the $4$th Burnside group of the welded link given in Proposition~\ref{prop-Whitehead} is isomorphic to that of the trivial $2$-component link. 
\end{remark}

\end{document}